\documentclass[11pt,reqno]{amsart}

\usepackage[T1]{fontenc}
\usepackage[utf8]{inputenc}
\usepackage{lmodern}
\usepackage{microtype}
\usepackage[a4paper,margin=1.15in]{geometry}
\usepackage{amsmath,amssymb,amsthm,mathtools}
\usepackage{mathrsfs}
\usepackage{enumitem}
\usepackage{array,booktabs,tabularx}
\usepackage{aliascnt}
\usepackage{xcolor}
\usepackage{hyperref}
\usepackage[nameinlink,capitalize,noabbrev]{cleveref}
\usepackage{amscd}

\definecolor{linkblue}{RGB}{20,63,110}
\definecolor{citegreen}{RGB}{35,92,67}
\hypersetup{
  colorlinks=true,
  linkcolor=linkblue,
  citecolor=citegreen,
  urlcolor=linkblue,
  pdfauthor={Nobuo Iida},
  pdftitle={The Three-Dimensional Symplectization Question for Overtwisted Contact Structures},
  pdfsubject={Contact and symplectic topology},
  pdfkeywords={contact structure, symplectization, overtwisted contact structure, invertible cobordism, proper homotopy, almost-complex structure}
}

\numberwithin{equation}{section}
\allowdisplaybreaks[1]
\newtheorem{theorem}{Theorem}[section]

\newaliascnt{proposition}{theorem}
\newtheorem{proposition}[proposition]{Proposition}
\aliascntresetthe{proposition}

\newaliascnt{lemma}{theorem}
\newtheorem{lemma}[lemma]{Lemma}
\aliascntresetthe{lemma}

\newaliascnt{corollary}{theorem}
\newtheorem{corollary}[corollary]{Corollary}
\aliascntresetthe{corollary}

\newaliascnt{claim}{theorem}

\aliascntresetthe{claim}

\theoremstyle{definition}
\newaliascnt{definition}{theorem}
\newtheorem{definition}[definition]{Definition}
\aliascntresetthe{definition}

\newaliascnt{question}{theorem}
\newtheorem{question}[question]{Question}
\aliascntresetthe{question}

\theoremstyle{remark}
\newaliascnt{remark}{theorem}
\newtheorem{remark}[remark]{Remark}
\aliascntresetthe{remark}

\newcommand{\R}{\mathbb{R}}

\newcommand{\Z}{\mathbb{Z}}
\newcommand{\cJ}{\mathcal{J}}
\newcommand{\cP}{\mathcal{P}}

\newcommand{\id}{\operatorname{id}}
\newcommand{\pr}{\operatorname{pr}}
\newcommand{\Vol}{\operatorname{Vol}}
\newcommand{\Diff}{\operatorname{Diff}}

\newcommand{\Spin}{\operatorname{Spin}}
\newcommand{\Spinc}{\operatorname{Spin}^{c}}
\newcommand{\divis}{\operatorname{div}}
\newcommand{\congsymp}{\cong_{\mathrm{symp}}}
\newcommand{\congcont}{\cong_{\mathrm{cont}}}

\newcommand{\PD}{\operatorname{PD}}
\newcommand{\Wh}{\operatorname{Wh}}

\newcommand{\SO}{\mathrm{SO}}
\newcommand{\U}{\mathrm{U}}

\title[Overtwisted symplectization rigidity]{The Three-Dimensional Symplectization Question for Overtwisted Contact Structures}
\author{Nobuo Iida}
\address{Kavli Institute for the Physics and Mathematics of the Universe (WPI), The University of Tokyo, 5-1-5 Kashiwanoha, Kashiwa, Chiba 277-8583, Japan}
\email{iidanobuo1224@g.ecc.u-tokyo.ac.jp}
\date{August 24, 2026}

\begin{document}

\begin{abstract}
We consider closed connected oriented three-manifolds equipped with positive cooriented contact structures. We prove that a symplectomorphism between their symplectizations forces the existence of an orientation-preserving diffeomorphism between the underlying three-manifolds under which the two contact structures are homotopic as oriented plane fields. Combining this formal rigidity with the classification of overtwisted contact structures, we show that two closed overtwisted contact three-manifolds have symplectomorphic symplectizations if and only if they are contactomorphic. Thus, we resolve the three-dimensional symplectization question when both contact structures are overtwisted.

\end{abstract}

\maketitle
\setcounter{tocdepth}{2}
\tableofcontents

\section{Introduction}

Let $Y$ be a closed connected oriented three-manifold and let $\xi\subset TY$ be a positive cooriented contact structure.  Choose a positive contact form $\alpha$ with
\[
  \xi=\ker\alpha,
  \qquad
  \alpha\wedge d\alpha>0.
\]
Its symplectization is the exact symplectic manifold
\begin{equation}\label{eq:symplectization-definition}
  S(Y,\alpha)
  =\bigl(\R_t\times Y,\lambda_\alpha=e^t\alpha,
  \omega_\alpha=d\lambda_\alpha\bigr).
\end{equation}
The exact symplectic isomorphism type\footnote{For exact symplectic manifolds equipped with specified primitives $(X_i,\lambda_i)$, $i=0,1$, an \emph{exact symplectomorphism} is a diffeomorphism $\Psi\colon X_0\to X_1$ such that $\Psi^*\lambda_1-\lambda_0=dH$ for some smooth function $H\colon X_0\to\R$.  The \emph{exact symplectic isomorphism type} is the equivalence class of the pair $(X,\lambda)$ under exact symplectomorphisms.} depends only on the cooriented contact structure.  Indeed, if $\alpha'=e^g\alpha$, then
\begin{equation}\label{eq:change-contact-form}
  (t,y)\longmapsto (t-g(y),y)
\end{equation}
pulls $e^t\alpha'$ back to $e^t\alpha$; see \cref{lem:strict-lift}.  We therefore write $S(Y,\xi)$ when the defining contact form is immaterial.

Courte constructed pairs of contact manifolds in every odd dimension at least five that are not diffeomorphic although their symplectizations are exact symplectomorphic \cite{Courte2014}.  Concretely, for every $n\ge1$, the nondiffeomorphic manifolds
\[
  L(7,1)\times S^{2n}
  \quad\text{and}\quad
  L(7,2)\times S^{2n}
\]
carry contact structures whose symplectizations are exact symplectomorphic.  This leaves the following problem.

\begin{question}[Three-dimensional symplectization question]\label{q:symplectization}
Do there exist closed contact three-manifolds $(Y_0,\xi_0)$ and $(Y_1,\xi_1)$ that are not contactomorphic although
\[
  S(Y_0,\xi_0)\congsymp S(Y_1,\xi_1)?
\]
(The symplectomorphism is not assumed to be exact.)
\end{question}
This is Problem 40 in McDuff--Salamon's textbook \cite{McDuffSalamon2017}.
The purpose of this paper is to settle \cref{q:symplectization} when both contact structures are overtwisted.

\begin{theorem}[Overtwisted symplectization rigidity]\label{thm:main-overtwisted}
Let $Y_0$ and $Y_1$ be closed connected oriented three-manifolds equipped with positive cooriented overtwisted contact structures $\xi_0$ and $\xi_1$, respectively.  Then
\begin{equation}\label{eq:main-iff}
  S(Y_0,\xi_0)\congsymp S(Y_1,\xi_1)
  \quad\Longleftrightarrow\quad
  (Y_0,\xi_0)\congcont (Y_1,\xi_1).
\end{equation}
Moreover, every coorientation-preserving contactomorphism lifts to a strict exact symplectomorphism\footnote{Here \emph{strict} means equality of the chosen Liouville forms: $\widetilde f^{\,*}\lambda_1=\lambda_0$, rather than merely $\widetilde f^{\,*}\lambda_1-\lambda_0$ being exact.} of the symplectizations.
\end{theorem}

The contact-topological input is Eliashberg's classification of overtwisted contact structures by their formal homotopy classes \cite{Eliashberg1989,Huang2013}.  The main work is the following statement, which does not assume overtwistedness.

\begin{theorem}[Formal symplectization rigidity]\label{thm:formal-rigidity}
Let $Y_i$, $i=0,1$, be closed connected oriented three-manifolds equipped with positive cooriented contact structures $\xi_i$.  Suppose there is a symplectomorphism
\[
  \Phi\colon S(Y_0,\xi_0)\longrightarrow S(Y_1,\xi_1).
\]
Then there is an orientation-preserving diffeomorphism
\[
  f\colon Y_0\longrightarrow Y_1
\]
such that
\begin{equation}\label{eq:formal-conclusion}
  [\xi_0]=[f^*\xi_1]\in\cP(Y_0),
\end{equation}
where $\cP(Y_0)$ denotes the set of homotopy classes of oriented two-plane fields on $Y_0$.
\end{theorem}

\subsection*{Outline of the proof}
The argument is assembled from the following steps.
\begin{enumerate}[label=\textup{(\roman*)},leftmargin=2.4em]
  \item The negative end of a symplectization has a finite-volume
  neighborhood, whereas every neighborhood of the positive end has infinite
  symplectic volume.  Hence every symplectomorphism preserves the two ends
  separately.

  \item If an end-order-preserving diffeomorphism is written as
  \[
    \Phi(t,y)=(a(t,y),b(t,y)),
  \]
  then
  \begin{equation}\label{eq:intro-proper-homotopy}
    H_s(t,y)
    =\bigl((1-s)a(t,y)+st,\ b((1-s)t,y)\bigr)
  \end{equation}
  is a proper homotopy from $\Phi$ to $\id_\R\times h$, where
  $h(y)=b(0,y)$.  At this stage $h$ is only a degree-$+1$ homotopy
  equivalence.

  \item The hypersurface $\Phi(\{0\}\times Y_0)$ cuts out an invertible
  cobordism whose natural homotopy equivalence is $h$.  Hausmann--Jahren show
  that an invertible cobordism is an $h$-cobordism
  \cite[Proposition~3.11]{HausmannJahren2018}.  Kwasik--Schultz prove
  Whitehead-torsion vanishing for $h$-cobordisms between geometric closed
  three-manifolds \cite{KwasikSchultz1992}; after Perelman's geometrization
  theorem, this applies to all closed oriented three-manifolds.  Thus $h$ is
  simple.

  \item Turaev proves that a simple homotopy equivalence between geometric
  three-manifolds is homotopic to a homeomorphism
  \cite[Theorem~I]{Turaev1988Homeomorphisms}; see also
  \cite[Theorem~1.6]{Turaev1988Classification}.  His term
  ``geometric'' includes hyperbolic, Seifert, Haken manifolds and their
  connected sums.  Perelman's proof of geometrization therefore makes this
  theorem applicable to all closed oriented three-manifolds.  It follows that
  $h$ is homotopic to an orientation-preserving diffeomorphism $f$.  Hence
  there is an end-order-preserving proper homotopy from $\Phi$ to
  $\id_\R\times f$.

  \item A compatible almost-complex structure on a symplectization determines
  the contact plane field on a standard slice.  Gompf's homeomorphism functor
  for homotopy classes of almost-complex structures on manifolds homeomorphic
  to $\R\times Y$ shows that properly homotopic diffeomorphisms induce the
  same pullback on these classes \cite[Theorems~4.1 and~4.7]{Gompf2022}.
  This yields \eqref{eq:formal-conclusion}.

  \item If both contact structures are overtwisted, Eliashberg's theorem
  upgrades the plane-field homotopy to contact isotopy, and Gray stability
  produces a contactomorphism.
\end{enumerate}

\subsection*{Organization of the paper}
\Cref{sec:conventions} fixes the orientation and proper-homotopy
conventions.  \Cref{sec:ends} proves end-order preservation by volume, and
\cref{sec:proper-product} constructs the explicit proper homotopy.
\Cref{sec:cobordism} identifies the invertible cobordism and its natural
marking.  \Cref{sec:realization} records prime decomposition and
geometrization, recalls the minimum Whitehead-torsion formalism, and applies
the cited theorems of Kwasik--Schultz and Turaev.
\Cref{sec:almost-complex} introduces the relevant almost-complex homotopy
classes and uses Gompf's cited functoriality theorem to compare them.  The
formal and overtwisted rigidity theorems are proved in
\cref{sec:formal-proof,sec:overtwisted-proof};
\cref{sec:scope} explains why the conclusion is special to dimension three.
Detailed reconstructions of the external three- and four-dimensional inputs
are deliberately omitted from the present paper.

\section{Conventions and preliminary facts}\label{sec:conventions}

All manifolds are smooth, Hausdorff, and second countable.  Unless otherwise stated, three-manifolds are closed, connected, and oriented.  Contact structures are positive and cooriented.  A contactomorphism is required to preserve the coorientation.

\subsection{Contact and symplectic orientations}

Let $(Y,\alpha)$ be a positive contact three-manifold.  On $\R_t\times Y$, we use the product orientation with $\partial_t$ first.  Since
\begin{equation}\label{eq:symplectic-form}
  \omega_\alpha=d(e^t\alpha)
  =e^t(dt\wedge\alpha+d\alpha),
\end{equation}
we have
\begin{equation}\label{eq:symplectic-volume-form}
  \frac{1}{2}\omega_\alpha^2
  =e^{2t}\,dt\wedge\alpha\wedge d\alpha.
\end{equation}
Thus the symplectic orientation agrees with the product orientation.  In particular, every symplectomorphism of symplectizations is orientation preserving.

The orientation of the contact plane field $\xi=\ker\alpha$ is the one induced by $d\alpha|_\xi$.  Together with the chosen coorientation, it recovers the orientation of $Y$.  Conversely, on an oriented three-manifold, an oriented plane field has a preferred positive normal line and hence a preferred coorientation.  We will therefore move freely between homotopies of positive cooriented plane fields and homotopies of oriented plane fields.

\subsection{Changing the defining form and lifting contactomorphisms}

The map \eqref{eq:change-contact-form} proves the independence of the symplectization from the defining contact form.  The same formula gives the converse direction in \cref{thm:main-overtwisted}.

\begin{lemma}[Strict lift of a contactomorphism]\label{lem:strict-lift}
Let
\[
  f\colon (Y_0,\xi_0)\longrightarrow (Y_1,\xi_1)
\]
be a coorientation-preserving contactomorphism.  For positive contact forms $\alpha_i$ defining $\xi_i$, write
\[
  f^*\alpha_1=e^g\alpha_0
\]
for a smooth function $g\colon Y_0\to\R$.  Then
\begin{equation}\label{eq:strict-lift}
  \widetilde f(t,y)=\bigl(t-g(y),f(y)\bigr)
\end{equation}
is a strict exact symplectomorphism satisfying
\[
  \widetilde f^{\,*}(e^t\alpha_1)=e^t\alpha_0.
\]
\end{lemma}

\begin{proof}
A direct calculation gives
\[
  \widetilde f^{\,*}(e^t\alpha_1)
  =e^{t-g(y)}f^*\alpha_1
  =e^{t-g(y)}e^{g(y)}\alpha_0
  =e^t\alpha_0.
\]
Applying the exterior derivative and using its naturality under pullback gives
\[
  \widetilde f^{\,*}d(e^t\alpha_1)
  =d\bigl(\widetilde f^{\,*}(e^t\alpha_1)\bigr)
  =d(e^t\alpha_0).
\]
Thus $\widetilde f$ preserves the symplectic forms.  Since it preserves the Liouville forms themselves, it is strict and, in particular, exact.
\end{proof}

\subsection{Standard ends and proper homotopies}

For a closed connected manifold $Y$ and $T>0$, set
\begin{equation}\label{eq:standard-ends}
  E_-^Y(T):=(-\infty,-T]\times Y,
  \qquad
  E_+^Y(T):=[T,\infty)\times Y.
\end{equation}

\begin{definition}\label{def:end-order}
A proper map
\[
  F\colon \R\times Y_0\longrightarrow \R\times Y_1
\]
is called \emph{end-order-preserving} if, for every $R>0$, there
exists $T>0$ such that
\[
  F\bigl(E_-^{Y_0}(T)\bigr)\subset E_-^{Y_1}(R)
  \qquad\text{and}\qquad
  F\bigl(E_+^{Y_0}(T)\bigr)\subset E_+^{Y_1}(R).
\]
It is called \emph{end-order-reversing} if, for every $R>0$, there
exists $T>0$ such that
\[
  F\bigl(E_-^{Y_0}(T)\bigr)\subset E_+^{Y_1}(R)
  \qquad\text{and}\qquad
  F\bigl(E_+^{Y_0}(T)\bigr)\subset E_-^{Y_1}(R).
\]
\end{definition}

\begin{lemma}[End-order dichotomy for product homeomorphisms]\label{lem:homeomorphism-end-order}
Let $Y_0$ and $Y_1$ be closed connected manifolds.  Every homeomorphism
\[
  F\colon \R\times Y_0\longrightarrow \R\times Y_1
\]
is either end-order-preserving or end-order-reversing.
\end{lemma}

\begin{proof}
First note that every homeomorphism is proper.  Indeed, if
$K\subset \R\times Y_1$ is compact, then
\[
  F^{-1}(K)
\]
is the image of $K$ under the continuous inverse map $F^{-1}$, and
hence is compact.

Fix $R>0$.  Since
\[
  F^{-1}\bigl([-R,R]\times Y_1\bigr)
\]
is compact, there exists $T>0$ such that
\[
  F^{-1}\bigl([-R,R]\times Y_1\bigr)
  \subset (-T,T)\times Y_0.
\]
Equivalently,
\[
  F\bigl(E_-^{Y_0}(T)\cup E_+^{Y_0}(T)\bigr)
  \subset
  \bigl(\R\times Y_1\bigr)\setminus([-R,R]\times Y_1).
\]
The latter complement has precisely two connected components,
\[
  (-\infty,-R)\times Y_1
  \qquad\text{and}\qquad
  (R,\infty)\times Y_1,
\]
because $Y_1$ is connected.  Since each of
$E_-^{Y_0}(T)$ and $E_+^{Y_0}(T)$ is connected, its image under $F$
must lie entirely in one of these two components.

The two source ends cannot both map into the same target component.
Indeed, suppose, for example, that
\[
  F\bigl(E_-^{Y_0}(T)\bigr)
  \cup
  F\bigl(E_+^{Y_0}(T)\bigr)
  \subset
  (R,\infty)\times Y_1.
\]
Since $F$ is surjective we have
\[
  E_-^{Y_1}(R)
  \subset
  F\bigl([-T,T]\times Y_0\bigr).
\]
The set on the right is compact.  Therefore its projection to the
$\R$-factor is compact, and in particular bounded.  On the other hand,
\[
  \pr_{\R}\bigl(E_-^{Y_1}(R)\bigr)=(-\infty,-R],
\]
which is unbounded.  This is a contradiction.
The case in which both source ends map into the negative target
component is identical.

It remains to see that the resulting choice does not depend on $R$.
Suppose $0<R_1<R_2$.  Choose corresponding cutoffs $T_1,T_2$ as above,
and enlarge them if necessary so that $T_2\ge T_1$.  Then
\[
  E_-^{Y_0}(T_2)\subset E_-^{Y_0}(T_1),
  \qquad
  E_+^{Y_0}(T_2)\subset E_+^{Y_0}(T_1).
\]
Hence a source end that is sent toward the negative target end for
the cutoff $R_1$ cannot be sent toward the positive target end for
the larger cutoff $R_2$, and similarly for the positive source end.
Therefore the same alternative holds for every $R>0$.

Consequently either
\[
  F\bigl(E_-^{Y_0}(T)\bigr)\subset E_-^{Y_1}(R),
  \qquad
  F\bigl(E_+^{Y_0}(T)\bigr)\subset E_+^{Y_1}(R)
\]
for sufficiently large $T$ and every $R>0$, or the two target ends
are interchanged.  These are precisely the two alternatives in
\cref{def:end-order}.
\end{proof}

A homotopy $H_s\colon X\to X'$, $s\in[0,1]$, is \emph{proper} if the combined map
\[
  H\colon [0,1]\times X\longrightarrow X',
  \qquad
  H(s,x)=H_s(x),
\]
is proper.

\section{Symplectic volume and the order of the ends}\label{sec:ends}

The only analytic input needed to control the ends is the elementary asymmetry of the symplectic volume form \eqref{eq:symplectic-volume-form}.

\begin{lemma}[Volumes of the standard ends]\label{lem:end-volumes}
Let $(Y,\alpha)$ be a closed positive contact three-manifold and put $\omega=d(e^t\alpha)$.  For every $T\in\R$,
\begin{align}
  \Vol_\omega\bigl((-\infty,T]\times Y\bigr)
  &=\frac{e^{2T}}{2}\int_Y\alpha\wedge d\alpha<\infty,
  \label{eq:negative-end-volume}\\
  \Vol_\omega\bigl([T,\infty)\times Y\bigr)
  &=\infty.
  \label{eq:positive-end-volume}
\end{align}
Here volume is computed using $\omega^2/2$.
\end{lemma}

\begin{proof}
By \eqref{eq:symplectic-volume-form},
\[
  \frac{1}{2}\omega^2=e^{2t}\,dt\wedge\alpha\wedge d\alpha.
\]
Since $Y$ is closed and $\alpha\wedge d\alpha$ is a positive volume form,
\[
  0<\int_Y\alpha\wedge d\alpha<\infty.
\]
The two assertions follow from
\[
  \int_{-\infty}^T e^{2t}\,dt=\frac{e^{2T}}{2},
  \qquad
  \int_T^\infty e^{2t}\,dt=\infty.
\]
\end{proof}

\begin{proposition}[End-order preservation]\label{prop:end-preservation}
For $i=0,1$, let $Y_i$ be a closed connected oriented three-manifold,
and let $\alpha_i$ be a positive contact form on $Y_i$.  Then every
symplectomorphism
\[
  \Phi\colon
  \bigl(\R\times Y_0,d(e^t\alpha_0)\bigr)
  \longrightarrow
  \bigl(\R\times Y_1,d(e^t\alpha_1)\bigr)
\]
is end-order-preserving.
\end{proposition}

\begin{proof}
Since a symplectomorphism is, in particular, a homeomorphism,
\cref{lem:homeomorphism-end-order} implies that $\Phi$ is either
end-order-preserving or end-order-reversing.  Suppose for contradiction
that it is end-order-reversing.  Then $\Phi^{-1}$ is also
end-order-reversing.  Applying \cref{def:end-order} to $\Phi^{-1}$ with target cutoff
$1$, we obtain $T>0$ such that
\[
  \Phi^{-1}\bigl(E_+^{Y_1}(T)\bigr)
  \subset E_-^{Y_0}(1)
  \subset E_-^{Y_0}(0).
\]
Equivalently,
\[
  E_+^{Y_1}(T)
  \subset \Phi\bigl(E_-^{Y_0}(0)\bigr).
\]
The set $E_-^{Y_0}(0)=(-\infty,0]\times Y_0$ has finite symplectic volume, whereas $E_+^{Y_1}(T)$ has infinite symplectic volume by \cref{lem:end-volumes}.  Since a symplectomorphism preserves the symplectic volume measure, this is impossible.  Hence $\Phi$ is end-order-preserving.
\end{proof}

Write an end-order-preserving proper map in product coordinates as
\begin{equation}\label{eq:coordinate-map}
  F(t,y)=\bigl(a(t,y),b(t,y)\bigr),
\end{equation}
where $a\colon\R\times Y_0\to\R$ and $b\colon\R\times Y_0\to Y_1$.

\begin{lemma}[Uniform divergence of the height component]\label{lem:uniform-divergence}
Let $F$ be an end-order-preserving proper map of the form \eqref{eq:coordinate-map}.  Then
\begin{equation}\label{eq:uniform-plus}
  \lim_{t\to+\infty}\inf_{y\in Y_0}a(t,y)=+\infty,
\end{equation}
and
\begin{equation}\label{eq:uniform-minus}
  \lim_{t\to-\infty}\sup_{y\in Y_0}a(t,y)=-\infty.
\end{equation}
Equivalently, for every $R>0$ there is $T>0$ such that, for all $y\in Y_0$,
\[
  t\ge T\Longrightarrow a(t,y)>R,
  \qquad
  t\le -T\Longrightarrow a(t,y)<-R.
\]
\end{lemma}

\begin{proof}
Apply \cref{def:end-order} with target cutoff $R+1$.  There is $T>0$ such that
\[
  F\bigl(E_+^{Y_0}(T)\bigr)\subset E_+^{Y_1}(R+1),
  \qquad
  F\bigl(E_-^{Y_0}(T)\bigr)\subset E_-^{Y_1}(R+1).
\]
Thus $a(t,y)>R$ for $t\ge T$ and $a(t,y)<-R$ for $t\le -T$, uniformly in $y$.  This is equivalent to \eqref{eq:uniform-plus} and \eqref{eq:uniform-minus}.
\end{proof}

\section{A proper product normal form}\label{sec:proper-product}

\begin{lemma}[Proper-homotopy invariance]\label{lem:proper-homotopy-invariance}
Let $X_0$ and $X_1$ be smooth manifolds, and let
\[
  H\colon[0,1]\times X_0\longrightarrow X_1
\]
be a proper homotopy with endpoints $H_0$ and $H_1$.  Then $H_0$ and $H_1$ induce the same pullback on compactly supported cohomology.  If $X_i=\R\times Y_i$ with $Y_i$ closed and connected, then $H_0$ and $H_1$ also have the same order behavior on the standard ends.  More precisely, if one endpoint is end-order-preserving or end-order-reversing, then the whole homotopy has the same property uniformly in its parameter.
\end{lemma}

\begin{proof}
A proper map extends over the one-point compactifications by sending the point at infinity to the point at infinity.  Properness of the combined map $H$ implies that these extensions form a based homotopy
\[
  H^+\colon[0,1]\times X_0^+\longrightarrow X_1^+.
\]
Under the standard identification
\[
  H_c^*(X)\cong\widetilde H^*(X^+),
\]
the based homotopy $H^+$ shows that $H_0$ and $H_1$ induce the same
pullback on compactly supported cohomology.

For the statement about the ends, fix $R\ge0$.  The inverse image
\[
  H^{-1}\bigl([-R,R]\times Y_1\bigr)
\]
is compact, so it is contained in $[0,1]\times(-T,T)\times Y_0$ for some $T$.  After increasing $T$ so that the prescribed end behavior holds for $H_0$, the connected sets
\[
  [0,1]\times E_+^{Y_0}(T),
  \qquad
  [0,1]\times E_-^{Y_0}(T)
\]
have images disjoint from $[-R,R]\times Y_1$.  Each image is therefore contained in a single component of its complement.  The component is determined at $s=0$, so the same assignment of ends holds for every $s\in[0,1]$, with the same $T$.
\end{proof}

Let
\[
  F\colon\R\times Y_0\longrightarrow\R\times Y_1,
  \qquad
  F(t,y)=\bigl(a(t,y),b(t,y)\bigr),
\]
be an end-order-preserving diffeomorphism.  Define the cross-sectional map
\begin{equation}\label{eq:definition-h}
  h\colon Y_0\longrightarrow Y_1,
  \qquad
  h(y)=b(0,y).
\end{equation}

\begin{proposition}[Explicit proper product reduction]\label{prop:proper-product-reduction}
The maps
\begin{equation}\label{eq:proper-homotopy-formula}
  H_s(t,y)
  =\bigl((1-s)a(t,y)+st,\ b((1-s)t,y)\bigr),
  \qquad 0\le s\le1,
\end{equation}
form an end-order-preserving proper homotopy satisfying
\[
  H_0=F,
  \qquad
  H_1=\id_\R\times h.
\]
In particular, $h$ is a homotopy equivalence.
\end{proposition}

\begin{proof}
At $s=0$, formula \eqref{eq:proper-homotopy-formula} gives
\[
  H_0(t,y)=\bigl(a(t,y),b(t,y)\bigr)=F(t,y),
\]
whereas at $s=1$ it gives
\[
  H_1(t,y)=\bigl(t,b(0,y)\bigr)=(\id_\R\times h)(t,y).
\]
We verify properness uniformly in $s$.  Let $K\subset\R\times Y_1$ be compact.  Choose $R>0$ such that
\[
  K\subset[-R,R]\times Y_1.
\]
By \cref{lem:uniform-divergence}, after increasing $T>R$ we may arrange that, for every $y\in Y_0$,
\[
  t\ge T\Longrightarrow a(t,y)>R,
  \qquad
  t\le -T\Longrightarrow a(t,y)<-R.
\]
For $t\ge T$, both $a(t,y)$ and $t$ are greater than $R$, so
\[
  (1-s)a(t,y)+st>R
\]
for every $s\in[0,1]$.  Similarly, for $t\le -T$ the same convex combination is less than $-R$.  Therefore
\[
  H^{-1}(K)
  \subset [0,1]\times[-T,T]\times Y_0.
\]
The set on the right is compact, and $H^{-1}(K)$ is closed in it.  Hence $H$ is proper.

The estimates also give the required uniform end inclusions.  More explicitly, for the chosen $R$ and $T$ they give, simultaneously for all $s\in[0,1]$,
\[
  H_s\bigl(E_+^{Y_0}(T)\bigr)\subset E_+^{Y_1}(R),
  \qquad
  H_s\bigl(E_-^{Y_0}(T)\bigr)\subset E_-^{Y_1}(R).
\]
Thus the homotopy is end-order-preserving in the uniform sense of \cref{def:end-order}.

Since $F$ is a diffeomorphism, it is a homotopy equivalence.  The endpoint $\id_\R\times h$ is homotopic to $F$ and hence is also a homotopy equivalence.  Composing with the standard deformation retractions of $\R\times Y_i$ onto $\{0\}\times Y_i$ shows that $h$ is a homotopy equivalence.
\end{proof}

The orientation of the cross-sectional map $h$ will be needed later.

\begin{lemma}[Degree of the cross-sectional map]\label{lem:degree-h}
If $F$ preserves the product orientations, then the homotopy equivalence $h$ in \eqref{eq:definition-h} has degree $+1$.
\end{lemma}

\begin{proof}
Let
\[
  [\R\times Y_i]_c\in H_c^4(\R\times Y_i;\Z)
\]
denote the compactly supported orientation class.  An orientation-preserving diffeomorphism pulls the positive generator back to the positive generator, so
\[
  F^*[\R\times Y_1]_c=[\R\times Y_0]_c.
\]
Properly homotopic maps induce the same pullback on compactly supported cohomology by \cref{lem:proper-homotopy-invariance}.  Hence \cref{prop:proper-product-reduction} gives
\[
  (\id_\R\times h)^*[\R\times Y_1]_c
  =[\R\times Y_0]_c.
\]
Let
\[
  \mu_i\in H^3(Y_i;\Z)
\]
denote the positive cohomological orientation class.  Under the
K\"unneth identification
\[
  H_c^4(\R\times Y_i;\Z)
  \cong H_c^1(\R;\Z)\otimes H^3(Y_i;\Z),
\]
the pullback by $\id_\R\times h$ is the identity on the first factor
and $h^*$ on the second.  It follows that
\[
  h^*\mu_1=\mu_0,
\]
and hence $\deg h=+1$.
\end{proof}

\section{Invertible cobordisms and the cross-sectional marking}
\label{sec:cobordism}\label{sec:invertible-cobordisms}

\begin{definition}\label{def:invertible-cobordism}
A smooth cobordism $(W,j_Y,j_{Y'})$ from a closed manifold $Y$ to a
closed manifold $Y'$ is \emph{invertible} if there is a cobordism from $Y'$
to $Y$ whose two composites with $W$ are product cobordisms, relative to
their outer boundary parametrizations.  It is an \emph{$h$-cobordism} if
both boundary inclusions are homotopy equivalences.

If $W$ is an $h$-cobordism from $Y$ to $Y'$ and
$q\colon W\to Y$ is a homotopy inverse of the incoming boundary inclusion,
then
\begin{equation}\label{eq:natural-equivalence}
  q\circ j_{Y'}\colon Y'\longrightarrow Y
\end{equation}
is called a \emph{natural homotopy equivalence} associated to $W$.  Its
homotopy class is independent of the chosen homotopy inverse.
\end{definition}

We use the following two results without reproving them.

\begin{theorem}[Hausmann--Jahren]\label{thm:HJ-invertible}
Let $Y$ and $Y'$ be closed smooth manifolds.
\begin{enumerate}[label=\textup{(\roman*)},leftmargin=2.4em]
  \item The open products $\R\times Y$ and $\R\times Y'$ are diffeomorphic
  if and only if there is an invertible smooth cobordism from $Y$ to $Y'$.
  \item Every invertible smooth cobordism is an $h$-cobordism.
\end{enumerate}
\end{theorem}

\begin{proof}[Source]
Part~\textup{(i)} is
\cite[Proposition~3.3]{HausmannJahren2018}, and part~\textup{(ii)} is
\cite[Proposition~3.11]{HausmannJahren2018}.
\end{proof}

We now identify the particular invertible cobordism associated to the
proper product reduction.  Let
\[
  F\colon\R\times Y_0\longrightarrow\R\times Y_1,
  \qquad F(t,y)=(a(t,y),b(t,y)),
\]
be an end-order-preserving diffeomorphism, and put
\begin{equation}\label{eq:sigma-slice}
  \Sigma:=F(\{0\}\times Y_0).
\end{equation}
Let
\[
  U_-:=F((-\infty,0]\times Y_0).
\]
By compactness of $\Sigma$ and end-order preservation of $F^{-1}$,
we may choose $r>0$ so large that
\[
  \Sigma\subset(-r,\infty)\times Y_1
\]
and
\begin{equation}\label{eq:deep-end-contained}
  (-\infty,-r]\times Y_1\subset U_-.
\end{equation}
By \cref{lem:uniform-divergence}, after increasing $r$ if necessary there
is $T>0$ such that
\begin{equation}\label{eq:far-negative-below-r}
  F((-\infty,-T]\times Y_0)
  \subset(-\infty,-r)\times Y_1.
\end{equation}
Define
\begin{equation}\label{eq:cobordism-W}
  W_F:=U_-\cap([-r,\infty)\times Y_1).
\end{equation}
Then $W_F$ is compact by \eqref{eq:far-negative-below-r}, and its boundary
is the disjoint union of $\{-r\}\times Y_1$ and $\Sigma$.  We regard it as a
cobordism from $Y_1$ to $Y_0$ by the markings
\begin{equation}\label{eq:boundary-markings}
  j_-(y)=(-r,y),
  \qquad
  j_+(y)=F(0,y).
\end{equation}

\begin{proposition}[The natural marking is the cross-sectional map]
\label{prop:natural-marking}
The marked cobordism $(W_F,j_-,j_+)$ is invertible.  Its natural homotopy
equivalence is represented by the cross-sectional map
\[
  h\colon Y_0\longrightarrow Y_1,
  \qquad h(y)=b(0,y),
\]
from \eqref{eq:definition-h}.
\end{proposition}

\begin{proof}
Choose $s>0$ and $v>0$ so that
\[
  \Sigma\subset(-\infty,s)\times Y_1,
  \qquad
  F^{-1}(\{s\}\times Y_1)\subset(0,v)\times Y_0,
  \qquad
  F(\{v\}\times Y_0)\subset(s,\infty)\times Y_1.
\]
Such choices follow from compactness and end-order preservation.  The three
compact regions bounded successively by
\[
  \{-r\}\times Y_1,
  \quad \Sigma,
  \quad \{s\}\times Y_1,
  \quad F(\{v\}\times Y_0)
\]
are exactly the regions used in the proof of
\cite[Proposition~3.3, implication \textup{(a)}$\Rightarrow$\textup{(b)}]{HausmannJahren2018}; that argument shows that the first region $W_F$ is
invertible.  By \cref{thm:HJ-invertible}, it is therefore an
$h$-cobordism.

Let
\[
  p:=\pr_{Y_1}|_{W_F}\colon W_F\longrightarrow Y_1.
\]
Since $p\circ j_-=\id_{Y_1}$ and $j_-$ is a homotopy equivalence, $p$ is a
homotopy inverse of $j_-$.  Consequently the associated natural homotopy
equivalence is
\[
  p\circ j_+\colon Y_0\longrightarrow Y_1.
\]
For $y\in Y_0$ this map has value
\[
  (p\circ j_+)(y)=\pr_{Y_1}(F(0,y))=b(0,y)=h(y),
\]
as required.
\end{proof}

\section{Three-dimensional topology and simple homotopy}
\label{sec:realization}

\subsection{Prime decomposition and geometrization}
\label{subsec:prime-geometrization}

We record the two three-dimensional decomposition results that determine
the scope of the theorems of Kwasik--Schultz and Turaev.

\begin{definition}
A closed connected three-manifold $P$ is \emph{prime} if every
connected-sum decomposition
\[
  P\cong P_1\#P_2
\]
has one summand homeomorphic to $S^3$.  It is \emph{irreducible} if
every smoothly embedded two-sphere in $P$ bounds a three-ball.  A
closed connected oriented irreducible three-manifold is
\emph{Haken} if it contains a closed embedded two-sided incompressible
surface other than $S^2$; here incompressibility means that the
inclusion induces an injection on fundamental groups.
\end{definition}

\begin{theorem}[Kneser--Milnor prime decomposition]
\label{thm:prime-decomposition}
For $q\ge0$, write
\[
  H_q:=\mathop{\#}\limits^{q}(S^1\times S^2),
  \qquad
  H_0:=S^3.
\]
Every closed connected oriented three-manifold $Y$ admits a
decomposition
\begin{equation}\label{eq:prime-decomposition}
  Y
  \cong
  P_1\#\cdots\#P_r\#H_q,
\end{equation}
where each $P_i$ is irreducible and is not homeomorphic to $S^3$.
The integers $r$ and $q$, and the homeomorphism types of the $P_i$,
are unique up to reordering.
\end{theorem}

\begin{proof}[Source]
The existence theorem is due to Kneser and the uniqueness theorem to
Milnor \cite{Kneser1929,Milnor1962}.  The displayed modern form also
uses the Poincar\'e theorem, which is contained in Perelman's proof of
geometrization.
\end{proof}

\begin{theorem}[Geometrization in the form used here]
\label{thm:geometrization-used}
Let $P$ be a closed connected oriented irreducible three-manifold.
If $\pi_1(P)$ is finite, then $P$ is spherical.  If $\pi_1(P)$ is
infinite, then there is a finite family of pairwise disjoint embedded
incompressible tori such that every component obtained by cutting
along these tori is Seifert fibered or has interior admitting a
complete finite-volume hyperbolic metric.
\end{theorem}

\begin{proof}[Source]
This is the closed orientable form of Thurston's geometrization
conjecture, proved by Perelman
\cite{Perelman2002,Perelman2003Surgery,Perelman2003Extinction};
see also \cite[Sections~4--6]{Scott1983} for the geometric and JSJ
formulation used here.
\end{proof}

\begin{definition}[The broad geometric class]
\label{def:broad-geometric}
A closed connected oriented three-manifold is called
\emph{geometric in the broad Thurston sense} if it is a finite
connected sum of hyperbolic, Seifert fibered, and Haken
three-manifolds.
\end{definition}

This is the broad class called geometric by Kwasik--Schultz and by
Turaev.  It should not be confused with the narrower condition that
the whole manifold admit one of Thurston's eight homogeneous
geometries.

\begin{corollary}\label{cor:all-3manifolds-geometric}
Every closed connected oriented three-manifold is geometric in the
broad Thurston sense.
\end{corollary}

\begin{proof}
Apply \cref{thm:prime-decomposition}.  The factor
$S^1\times S^2$ is Seifert fibered.  Let $P_i$ be one of the
irreducible factors.  If $P_i$ is Haken, it already belongs to the
class in \cref{def:broad-geometric}.  If it is not Haken, then it
contains no incompressible torus.  Thus the torus family in
\cref{thm:geometrization-used} is empty.  The manifold $P_i$ is
therefore hyperbolic or Seifert fibered when its fundamental group is
infinite, and is spherical when its fundamental group is finite.
Every closed oriented spherical three-manifold is Seifert fibered;
see \cite[Theorem~4.10 and p.~452]{Scott1983}.  Hence every summand in
\eqref{eq:prime-decomposition} is hyperbolic, Seifert fibered, or
Haken.
\end{proof}

\subsection{Whitehead torsion}
\label{sec:whitehead-torsion}

Let $\pi$ be a group.  Its Whitehead group is
\begin{equation}\label{eq:whitehead-group}
  \Wh(\pi)
  :=
  K_1(\Z[\pi])\big/\langle[\pm g]\mid g\in\pi\rangle.
\end{equation}
Here $K_1(R)$ may be defined as the stable general linear group over
$R$ modulo elementary matrices.  For a homotopy equivalence
\[
  f\colon X\longrightarrow Z
\]
of finite connected CW complexes, Whitehead defined a torsion
\[
  \tau(f)\in\Wh(\pi_1Z).
\]
The map $f$ is called \emph{simple} when $\tau(f)=0$.  We use only
homotopy invariance and the standard formulas
\begin{align}
  \tau(v\circ u)&=\tau(v)+v_*\tau(u),
  \label{eq:torsion-composition}\\
  \tau(u^{-1})&=-u^{-1}_*\tau(u).
  \label{eq:torsion-inverse}
\end{align}
These are \cite[(22.4)--(22.5), pp.~72--73]{Cohen1973}.

An $h$-cobordism is called an \emph{$s$-cobordism} when a boundary
inclusion, equivalently both boundary inclusions, has zero Whitehead
torsion.

\subsection{The Kwasik--Schultz vanishing theorem}
\label{subsec:KS-vanishing}

\begin{theorem}[Kwasik--Schultz, after geometrization]
\label{thm:KS-vanishing}
Every smooth $h$-cobordism between closed connected oriented
three-manifolds is an $s$-cobordism.
\end{theorem}

\begin{proof}[Source and scope]
The main theorem of Kwasik--Schultz states this for oriented
three-manifolds that are geometric in the broad Thurston sense
\cite[p.~736]{KwasikSchultz1992}.  In the remark immediately following
the theorem, they characterize that class as the finite connected
sums of hyperbolic, Seifert fibered, and Haken manifolds
\cite[p.~736]{KwasikSchultz1992}.  The general statement above now
follows from \cref{cor:all-3manifolds-geometric}.
\end{proof}

\begin{corollary}[Simplicity of the natural marking]
\label{cor:natural-marking-simple}
Let $(W,j_-,j_+)$ be an $h$-cobordism between closed connected
oriented three-manifolds, and let $q$ be a homotopy inverse of $j_-$.
Then the natural homotopy equivalence
\[
  q\circ j_+
\]
is simple.
\end{corollary}

\begin{proof}
By \cref{thm:KS-vanishing},
\[
  \tau(j_-)=\tau(j_+)=0.
\]
The inverse formula gives $\tau(q)=0$, and the composition formula
gives
\[
  \tau(q\circ j_+)
  =
  \tau(q)+q_*\tau(j_+)
  =
  0.
\]
\end{proof}

\subsection{Turaev's realization theorem after geometrization}
\label{subsec:Turaev-realization}

Turaev uses the word \emph{geometric} in the broad Thurston sense.
Thus, in Turaev's terminology, a closed connected oriented
three-manifold is geometric if it is hyperbolic, Seifert fibered,
Haken, or a connected sum of manifolds of these types.  This does
not mean that the manifold itself admits one of Thurston's eight
homogeneous geometries.

\begin{theorem}[Turaev]
\label{thm:Turaev-geometric-realization}
Let $Y$ and $Y'$ be geometric three-manifolds in the broad Thurston
sense, and let
\[
  f\colon Y\longrightarrow Y'
\]
be a map of degree $+1$.  Then $f$ is homotopic to an
orientation-preserving PL homeomorphism if and only if it is a simple
homotopy equivalence.
\end{theorem}
\begin{proof}[Source]
Turaev proves that a map between geometric three-manifolds is homotopic
to a PL homeomorphism if and only if it is a simple homotopy
equivalence; see \cite[Theorem~I]{Turaev1988Homeomorphisms}.  The same
map-level statement appears as
\cite[Theorem~1.6]{Turaev1988Classification}.  Since the map in the
present formulation has degree $+1$, any homeomorphism in its homotopy
class is orientation preserving.
\end{proof}

\begin{corollary}[Turaev realization for all closed oriented
three-manifolds]
\label{cor:Turaev-all-3manifolds}
Let
\[
  f\colon Y\longrightarrow Y'
\]
be an orientation-preserving simple homotopy equivalence between
closed connected oriented three-manifolds.  Then $f$ is homotopic to
an orientation-preserving diffeomorphism.
\end{corollary}

\begin{proof}
By \cref{cor:all-3manifolds-geometric}, both $Y$ and $Y'$ are
geometric in the broad Thurston sense.  Hence
\cref{thm:Turaev-geometric-realization} shows that $f$ is homotopic to
an orientation-preserving PL homeomorphism
\[
  h\colon Y\longrightarrow Y'.
\]
By Moise's uniqueness theorem for three-dimensional PL and smooth
structures, together with smoothing of PL homeomorphisms, $h$, and
hence $f$, is homotopic to an orientation-preserving diffeomorphism
\cite{Moise1952,Munkres1960}.
\end{proof}

\begin{theorem}[Realization of the natural marking]
\label{thm:KST-realization}
Let $W$ be an invertible cobordism between closed connected oriented
three-manifolds.  Every orientation-preserving natural homotopy
equivalence associated to $W$ is homotopic to an
orientation-preserving diffeomorphism.
\end{theorem}

\begin{proof}
By \cref{thm:HJ-invertible}, $W$ is an $h$-cobordism.  Its natural
homotopy equivalence is simple by
\cref{cor:natural-marking-simple}, and hence is realized by a
diffeomorphism by \cref{cor:Turaev-all-3manifolds}.
\end{proof}

\begin{proposition}[Product normal form up to proper homotopy]
\label{prop:product-normal-form}
Let $Y_0$ and $Y_1$ be closed connected oriented three-manifolds, and
let
\[
  F\colon\R\times Y_0\longrightarrow\R\times Y_1
\]
be an orientation-preserving end-order-preserving diffeomorphism.
Then there is an orientation-preserving diffeomorphism
\[
  f\colon Y_0\longrightarrow Y_1
\]
such that $F$ is properly homotopic to $\id_\R\times f$ through end-order-preserving proper maps.
\end{proposition}

\begin{proof}
By \cref{prop:natural-marking}, the cross-sectional map
\[
  h\colon Y_0\longrightarrow Y_1
\]
is a natural homotopy equivalence of the invertible cobordism $W_F$.
By \cref{lem:degree-h}, it has degree $+1$.  Hence
\cref{thm:KST-realization} gives an orientation-preserving
diffeomorphism
\[
  f\colon Y_0\longrightarrow Y_1
\]
homotopic to $h$.

Let
\[
  k_s\colon Y_0\longrightarrow Y_1,
  \qquad
  0\le s\le1,
\]
be a homotopy from $h$ to $f$.  Then
\[
  K_s(t,y)=(t,k_s(y))
\]
is an end-order-preserving proper homotopy.  Concatenating it with the
proper homotopy of \cref{prop:proper-product-reduction} proves the
claim.
\end{proof}

\section{Almost-complex structures and oriented plane fields}
\label{sec:almost-complex}

We now explain why the proper homotopy in
\cref{prop:product-normal-form} is sufficient to recover the full
homotopy class of the contact plane field.

\subsection{Almost-complex homotopy classes}

\begin{definition}[Almost-complex structures and the set $\cJ(V)$]
\label{def:J-of-V}
Let $V$ be an oriented smooth four-manifold.  A \emph{positive
almost-complex structure} on $V$ is a smooth bundle automorphism
\[
  J\colon TV\longrightarrow TV
\]
satisfying $J^2=-\id_{TV}$ and such that the complex orientation
determined by $J$ agrees with the given orientation of $V$.

If $g$ is a Riemannian metric on $V$, an almost-complex structure $J$
is called \emph{$g$-orthogonal}, or simply \emph{orthogonal} when the
metric is understood, if
\[
  g(Ju,Jv)=g(u,v)
\]
for all tangent vectors $u$ and $v$.  Equivalently, each fiber map
\[
  J_x\colon (T_xV,g_x)\longrightarrow(T_xV,g_x)
\]
is an isometry.

We write
\begin{equation}\label{eq:J-of-V-definition}
  \cJ(V)
  :=
  \frac{\{\text{positive almost-complex structures on }TV\}}
       {\text{homotopy through positive almost-complex structures}}.
\end{equation}
Thus an element of $\cJ(V)$ is a homotopy class $[J]$; no
integrability condition is imposed.
\end{definition}

An orientation-preserving diffeomorphism
\[
  F\colon V_0\longrightarrow V_1
\]
acts on representatives by geometric pullback,
\begin{equation}\label{eq:geometric-pullback}
  F^*J=(dF)^{-1}\circ J\circ dF,
\end{equation}
and hence induces
\[
  F^*\colon\cJ(V_1)\longrightarrow\cJ(V_0).
\]

For a closed oriented three-manifold $Y$, let $\cP(Y)$ denote the set
of homotopy classes of oriented two-plane fields in $TY$.  The
canonical bijection between $\cJ(\R\times Y)$ and $\cP(Y)$ will be
proved in \cref{subsec:product-correspondence}.

\subsection{The product correspondence}\label{subsec:product-correspondence}

\begin{proposition}[Almost-complex/plane-field correspondence]\label{prop:product-correspondence}
For every closed oriented three-manifold $Y$, there is a canonical bijection
\begin{equation}\label{eq:rho-bijection}
  \rho_Y\colon\cJ(\R\times Y)\xrightarrow{\cong}\cP(Y).
\end{equation}
For a representative $J$, the associated plane field on the slice $\{0\}\times Y$ is
\begin{equation}\label{eq:complex-tangency-plane}
  \rho_Y([J])=[TY\cap J(TY)].
\end{equation}
The bijection is natural for product diffeomorphisms: if $f\colon Y_0\to Y_1$ is orientation preserving, then
\begin{equation}\label{eq:rho-naturality}
  \rho_{Y_0}\bigl((\id_\R\times f)^*[J]\bigr)
  =f^*\rho_{Y_1}([J]).
\end{equation}
\end{proposition}

\begin{proof}
Choose a Riemannian metric on $Y$ and equip $\R\times Y$ with the product metric.  The inclusion of positive orthogonal almost-complex structures into all positive almost-complex structures is a fiberwise homotopy equivalence, so it suffices to work with orthogonal structures.

For an orthogonal positive almost-complex structure $J$, the vector $J\partial_t$ is a unit vector tangent to $Y$, since it is orthogonal to $\partial_t$.  Define
\[
  \xi_J=(J\partial_t)^\perp\subset TY.
\]
For $v\in TY$,
\[
  Jv\in TY
  \quad\Longleftrightarrow\quad
  \langle Jv,\partial_t\rangle=0
  \quad\Longleftrightarrow\quad
  \langle v,J\partial_t\rangle=0.
\]
Thus
\[
  \xi_J=TY\cap J(TY).
\]
The subbundle is $J$-invariant and hence oriented by $J$.  A homotopy of almost-complex structures gives a homotopy of the resulting oriented plane fields.

Conversely, let $\xi\subset TY$ be an oriented plane field.  Let $n$ be its positive unit normal, characterized by the condition that $(n,v_1,v_2)$ is positively oriented in $TY$ whenever $(v_1,v_2)$ is a positive basis of $\xi$.  Choose the positive orthogonal complex structure on the oriented rank-two bundle $\xi$, and define
\begin{equation}\label{eq:J-from-plane-field}
  J\partial_t=n,
  \qquad
  Jn=-\partial_t,
  \qquad
  J(\xi)=\xi.
\end{equation}
The spaces of auxiliary metrics and positive complex structures on an oriented plane bundle are contractible.  The two constructions are therefore inverse on homotopy classes and give a canonical bijection.

It remains to verify naturality.  Let
\[
  f\colon Y_0\longrightarrow Y_1
\]
be an orientation-preserving diffeomorphism, and put
\[
  F=\id_{\R}\times f\colon \R\times Y_0\longrightarrow \R\times Y_1.
\]
Let $J$ be a positive almost-complex structure on $\R\times Y_1$.
By definition, the pullback almost-complex structure $F^*J$ is
\[
  (F^*J)_x
  =
  (dF_x)^{-1}\circ J_{F(x)}\circ dF_x.
\]
Since $F$ preserves the $\R$-coordinate, it maps the slice
$\{0\}\times Y_0$ diffeomorphically onto $\{0\}\times Y_1$, and
\[
  dF(TY_0)=TY_1.
\]
Moreover,
\[
  dF\circ(F^*J)=J\circ dF.
\]
Consequently,
\begin{align*}
  dF\bigl(TY_0\cap(F^*J)(TY_0)\bigr)
  &=
  dF(TY_0)\cap dF\bigl((F^*J)(TY_0)\bigr)\\
  &=
  TY_1\cap J\bigl(dF(TY_0)\bigr)\\
  &=
  TY_1\cap J(TY_1).
\end{align*}
Thus, as plane subbundles of $TY_0$,
\[
  TY_0\cap(F^*J)(TY_0)
  =
  f^*\bigl(TY_1\cap J(TY_1)\bigr).
\]
This equality also respects orientations.  Indeed, the plane
$TY_0\cap(F^*J)(TY_0)$ is oriented by $F^*J$, while
$TY_1\cap J(TY_1)$ is oriented by $J$, and the identity
\[
  dF\circ(F^*J)=J\circ dF
\]
shows that the restriction of $dF$ to these rank-two bundles is
complex linear.  Hence it preserves their complex orientations.

Using \eqref{eq:complex-tangency-plane}, we therefore obtain
\[
  \rho_{Y_0}\bigl(F^*[J]\bigr)
  =
  f^*\rho_{Y_1}([J]),
\]
which is precisely \eqref{eq:rho-naturality}.

\end{proof}

\subsection{The adapted compatible structure}

Let $(Y,\alpha)$ be a positive contact three-manifold.  Recall that
the Reeb vector field $R_\alpha$ is uniquely characterized by
\[
  \alpha(R_\alpha)=1,
  \qquad
  \iota_{R_\alpha}d\alpha=0.
\]
Choose a complex structure $j$ on $\xi=\ker\alpha$ such that
\[
  d\alpha(v,jv)>0
  \qquad (0\ne v\in\xi).
\]
Define an $\R$-invariant almost-complex structure $J_\alpha$ on $\R\times Y$ by
\begin{equation}\label{eq:adapted-J}
  J_\alpha\partial_t=R_\alpha,
  \qquad
  J_\alpha R_\alpha=-\partial_t,
  \qquad
  J_\alpha|_\xi=j.
\end{equation}

\begin{lemma}[Adapted structures]\label{lem:adapted-J}
The almost-complex structure $J_\alpha$ is compatible with $\omega_\alpha=d(e^t\alpha)$, and
\begin{equation}\label{eq:adapted-plane}
  \rho_Y([J_\alpha])=[\xi].
\end{equation}
\end{lemma}

\begin{proof}
Write a tangent vector as
\[
  v=a\partial_t+bR_\alpha+w,
  \qquad w\in\xi.
\]
Then
\[
  J_\alpha v=aR_\alpha-b\partial_t+jw.
\]
The splitting
\[
  T(\R\times Y)
  =\langle\partial_t,R_\alpha\rangle\oplus\xi
\]
is $\omega_\alpha$-orthogonal and $J_\alpha$-invariant.  On
$\langle\partial_t,R_\alpha\rangle$, the map $J_\alpha$ preserves
$dt\wedge\alpha$.  On the oriented real two-plane bundle $\xi$, the
condition $d\alpha(v,jv)>0$ implies
\[
  d\alpha(jv,jw)=d\alpha(v,w)
  \qquad (v,w\in\xi).
\]
Consequently, $J_\alpha$ preserves $\omega_\alpha$.  Moreover, using
\eqref{eq:symplectic-form}, $\alpha(R_\alpha)=1$, and
$\iota_{R_\alpha}d\alpha=0$, we obtain
\[
  \omega_\alpha(v,J_\alpha v)
  =e^t\bigl(a^2+b^2+d\alpha(w,jw)\bigr)>0
\]
for $v\ne0$.  Hence $J_\alpha$ is $\omega_\alpha$-compatible.  Finally,
\[
  TY=\R R_\alpha\oplus\xi,
  \qquad
  J_\alpha(TY)=\R\partial_t\oplus\xi,
\]
so their intersection is exactly $\xi$, with its contact orientation.  This proves \eqref{eq:adapted-plane}.
\end{proof}

\subsection{Gompf's homeomorphism functor}
\label{subsec:Gompf-functor}

We next recall the invariants used by Gompf to classify
almost-complex structures on the open four-manifolds relevant to us.
Throughout this subsection, $Y$ denotes a closed connected oriented
three-manifold, and
\[
  V=\R\times Y
\]
is given the product orientation.

\subsubsection*{The primary invariant $\Gamma$.}

Let $J$ be a positive almost-complex structure on $V$, and let $s$ be
a spin structure on $V$.  Denote by $\mathfrak s_J$ the $\Spinc$
structure induced by $J$ and by $\mathfrak s_s$ the $\Spinc$ structure
obtained from $s$.  Since $\Spinc(V)$ is an affine space over
$H^2(V;\Z)$, there is a unique class
\begin{equation}\label{eq:Gamma-Gompf}
  \Gamma(J,s)\in H^2(V;\Z)
\end{equation}
satisfying
\[
  \mathfrak s_J=\mathfrak s_s+\Gamma(J,s).
\]
The determinant line of $\mathfrak s_s$ is trivial, whereas that of
$\mathfrak s_J$ has first Chern class $c_1(TV,J)$.  Hence
\begin{equation}\label{eq:two-Gamma}
  2\Gamma(J,s)=c_1(TV,J).
\end{equation}
This is Gompf's primary invariant
\cite[Definition~3.1]{Gompf2022}.  In particular, $\Gamma$ retains
possible two-torsion information that is invisible in
$c_1(TV,J)$ alone.

\subsubsection*{Framed representatives and the canonical $\Z$-action.}

We first fix the Poincar\'e-duality convention for the noncompact
manifold $V$.  Let
\[
  V^+=V\cup\{\infty\}
\]
be its one-point compactification, and write
\[
  H_k^{\mathrm{lf}}(V;\Z)
  := H_k(V^+,\{\infty\};\Z)
  \cong \widetilde H_k(V^+;\Z)
\]
for locally finite, or Borel--Moore, homology.  If
\[
  i\colon F\hookrightarrow V
\]
is a proper smooth embedding of an oriented surface without boundary,
then $F$, which may be noncompact, has a Borel--Moore fundamental class
\[
  [F]_{\mathrm{lf}}^{F}\in H_2^{\mathrm{lf}}(F;\Z).
\]
Locally finite homology is covariantly functorial for proper maps, so
we set
\[
  [F]_{\mathrm{lf}}:=i_*[F]_{\mathrm{lf}}^{F}
  \in H_2^{\mathrm{lf}}(V;\Z).
\]
Since $V$ is an oriented four-manifold, noncompact Poincar\'e duality
gives a natural isomorphism
\begin{equation}\label{eq:BM-PD}
  \PD_V\colon
  H_2^{\mathrm{lf}}(V;\Z)
  \xrightarrow{\cong}
  H^2(V;\Z).
\end{equation}

Let $\Omega(V)$ be the set of proper framed-cobordism classes of
properly embedded oriented surfaces without boundary in $V$.  Since a
smoothly embedded surface has a normal bundle
\[
  \nu F=TV|_F/TF,
\]
and properness makes its image closed in $V$, the tubular-neighborhood
theorem applies globally.  A framing trivializes $\nu F$, so the
Pontryagin--Thom collapse construction is available.  It gives a
canonical bijection
\begin{equation}\label{eq:PT-Omega}
  [V,S^2]\xrightarrow{\cong}\Omega(V).
\end{equation}
Indeed, if $u\colon V\to S^2$ is smooth and $p\in S^2$ is a regular
value, then $u^{-1}(p)$ is closed in $V$ and hence properly embedded;
the derivative of $u$ frames its normal bundle.  Conversely, a framed
tubular neighborhood of a proper surface can be collapsed fiberwise
to $D^2/\partial D^2\cong S^2$, with its complement sent to the
basepoint.

Forgetting the framing defines
\begin{equation}\label{eq:eta-framed-surface}
  \eta\colon\Omega(V)\longrightarrow H^2(V;\Z),
  \qquad
  \eta(\phi)=\PD_V([F]_{\mathrm{lf}}),
\end{equation}
where $F$ is any framed representative of $\phi$.  Equivalently, if
$\phi$ corresponds under \eqref{eq:PT-Omega} to a map
$u\colon V\to S^2$, then
\begin{equation}\label{eq:eta-pullback-generator}
  \eta(\phi)=u^*\iota,
\end{equation}
where $\iota\in H^2(S^2;\Z)$ is the positive generator.

Every oriented three-manifold is parallelizable, so $V=\R\times Y$
is parallelizable.  Choose a trivialization
\[
  \tau\colon TV\xrightarrow{\cong}V\times\R^4
\]
inducing the spin structure $s$, and choose a metric for which $\tau$
is orthonormal.  After deforming $J$ through positive
almost-complex structures, we may assume that it is orthogonal.  For
$x\in V$, put
\[
  \widehat J_x
  :=\tau_x\circ J_x\circ\tau_x^{-1}.
\]
The group $\SO(4)$ acts transitively by conjugation on the positive
orthogonal complex structures on $\R^4$, and the stabilizer of the
standard structure is $\U(2)$.  Thus these structures form
\[
  \SO(4)/\U(2)\cong S^2,
\]
and $J$ determines the map
\begin{equation}\label{eq:u-J-map}
  u_J\colon V\longrightarrow\SO(4)/\U(2)\cong S^2,
  \qquad
  u_J(x)=[\widehat J_x].
\end{equation}
Let $\phi_J\in\Omega(V)$ be its Pontryagin--Thom class.  Gompf proves
that
\begin{equation}\label{eq:Gamma-eta}
  \Gamma(J,s)=\eta(\phi_J);
\end{equation}
see \cite[Proposition~3.2]{Gompf2022}.  More generally, for fixed
$J$ and $s$, a class $\phi\in\Omega(V)$ is said to be
\emph{dual to $\Gamma(J,s)$} when
\begin{equation}\label{eq:phi-dual-Gamma}
  \eta(\phi)=\Gamma(J,s).
\end{equation}

We now describe Gompf's canonical $\Z$-action in the product model.
Using the homotopy equivalence $Y=\{0\}\times Y\hookrightarrow V$,
represent a class in $[V,S^2]$ by
\[
  u(t,y)=\bar u(y).
\]
Choose an embedded closed three-ball $B^3\subset Y$, and homotope
$\bar u$ so that it is constant with value $p_0$ on $B^3$.  Choose a
map of pairs
\[
  \eta_H\colon(B^3,\partial B^3)\longrightarrow(S^2,p_0)
\]
whose quotient map $S^3=B^3/\partial B^3\to S^2$ represents the
chosen positive generator of $\pi_3(S^2)\cong\Z$.  We fix the sign by
requiring its Pontryagin--Thom representative to be a $+1$-framed
unknot in $B^3$.  Replace $\bar u|_{B^3}$ by $\eta_H$ and leave
$\bar u$ unchanged outside $B^3$.  Pulling the resulting map back to
$V$ defines the positive generator of the action; replacing
$\eta_H$ by a representative of $n$ times this generator defines the
action of $n\in\Z$.

On the framed-surface side, if $K=\eta_H^{-1}(p)$ for a regular value
$p\ne p_0$, the added component is the proper framed cylinder
\[
  \R\times K\subset\R\times B^3\subset V,
\]
with the product framing induced by the $+1$-framing of the unknot
$K$.  By \cite[Proposition~3.4 and its proof]{Gompf2022}, this
construction defines a canonical $\Z$-action on $\cJ(V)$, independent
of the auxiliary choices, and for every fixed spin structure $s$ the
primary invariant induces a bijection
\begin{equation}\label{eq:Gamma-orbit-space}
  \cJ(V)/\Z\xrightarrow{\cong}H^2(V;\Z).
\end{equation}
Thus $\Gamma$ determines the $\Z$-orbit, while the secondary invariant
below distinguishes its elements.

\begin{remark}\label{rem:Gompf-KM-action}
Under the canonical correspondence
\[
  \rho_Y\colon\cJ(\R\times Y)\xrightarrow{\cong}\cP(Y),
\]
this is the same local $\pi_3(S^2)$-action that occurs in the absolute
grading of monopole Floer homology, but the conventions for the
positive generator are opposite.  Kronheimer--Mrowka define the
action of $n\in\Z$ on $\cP(Y)$ by modifying a plane field inside a
three-ball using a map
\[
  \varrho\colon(B^3,\partial B^3)\longrightarrow(\SO(3),1)
\]
of degree $-2n$; see
\cite[Definition~3.1.2, p.~51]{KronheimerMrowka2007}.  They identify
this $\Z$-set with the monopole Floer grading set in
\cite[Proposition~23.1.8, p.~454]{KronheimerMrowka2007}.

To compare signs, let
\[
  p\colon\SO(3)\longrightarrow S^2,
  \qquad p(A)=A(e_3),
\]
for a fixed positively oriented unit vector $e_3$.  With the standard
orientations, a generator of $\pi_3(\SO(3))$ represented by a map of
ordinary degree $+2$ projects to the positive Hopf generator fixed above.  Thus
the Kronheimer--Mrowka positive generator, defined using degree $-2$,
corresponds to the negative Hopf generator.  Gompf's positive
generator is, by our convention above, the positive Hopf generator,
equivalently the $+1$-framed unknot.  Hence
\[
  1_{\mathrm{Gompf}}=-1_{\mathrm{KM}}.
\]
The two actions have the same orbits but opposite chosen positive
directions.
\end{remark}

\subsubsection*{The secondary invariant $\widetilde\Theta$.}

The invariant $\Gamma$ does not distinguish the elements within a
single $\Z$-orbit.  Fix $J$, a spin structure $s$, and a class
$\phi\in\Omega(V)$ dual to $\Gamma(J,s)$.  Choose a framed surface
$F\subset V$ representing $\phi$.  Gompf proves that there is a
$J$-complex trivialization $\tau_J$ of $TV$ over the complement of a
tubular neighborhood of $F$, compatible with $s$, whose restriction
to every positive meridian of $F$ represents twice the generator of
\[
  \pi_1(\U(2))\cong\Z;
\]
see \cite[Proposition~3.2]{Gompf2022}.

Choose a closed oriented three-manifold $N\subset V$ separating the
two ends, and let $V_N$ be the closure of the component of
$V\setminus N$ containing the positive end.  Gompf chooses a compact
smooth four-manifold $Z^*$ with an orientation-reversing boundary
identification
\[
  \partial Z^*\cong\partial V_N
\]
so that, on
\[
  Z=Z^*\cup_{\partial V_N}V_N,
\]
the structure $J|_{V_N}$ extends to an almost-complex structure
$J_Z$.  The existence of this cap and extension is established in the
first paragraph of the proof of
\cite[Theorem~3.7]{Gompf2022}.

The relative first Chern class $c_1(J_Z,\tau_J)$ is represented by a
properly embedded oriented surface $\widetilde F\subset Z$ that, near
the positive end, consists of two parallel copies of $F$ determined
by $\phi$.  Let
\[
  e(\nu\widetilde F,\phi)
\]
be its relative normal Euler number, computed with respect to the
framing at infinity.  Set
\[
  d(J):=\divis c_1(TV,J),
\]
where divisibility is computed modulo torsion and is defined to be
zero when $c_1(TV,J)$ is torsion.  Gompf defines
\begin{equation}\label{eq:Gompf-Theta-definition}
  \widetilde\Theta(J,s,\phi)
  := e(\nu\widetilde F,\phi)-2\chi(Z)-3\sigma(Z)
  \in\Z/(4d(J)),
\end{equation}
with the convention $\Z/(0)=\Z$; see
\cite[Definition~3.6]{Gompf2022}.

Gompf proves that this class is independent of the cut, cap,
extension, representative, and compatible trivialization.  Together
with $\Gamma$, it is a complete invariant for $\cJ(V)$; see
\cite[Theorem~3.7]{Gompf2022}.  If $1\cdot J$ denotes the result of
applying the positive generator of the canonical $\Z$-action, then
\begin{equation}\label{eq:Gompf-Theta-action}
  \widetilde\Theta(1\cdot J,s,\phi)
  =\widetilde\Theta(J,s,\phi)-4.
\end{equation}

Identify a spin structure $s$ on $V=\R\times Y$ with its restriction
to the slice $\{0\}\times Y$.  Gompf's coset formula is
\begin{equation}\label{eq:Gompf-Theta-coset}
  \widetilde\Theta(J,s,\phi)
  \equiv 2\bigl(1+b_1(Y)\bigr)-\mu(Y,s)
  \pmod{4},
\end{equation}
where $\mu(Y,s)\in\Z/16$ is the Rohlin invariant; see
\cite[Proposition~4.5]{Gompf2022}.  Its reduction modulo $8$ is
preserved by orientation-preserving homeomorphisms
\cite[Proposition~4.4]{Gompf2022}.

We can now state the functoriality theorem in the form used below.

\begin{theorem}[Gompf]\label{thm:Gompf-homeomorphism-functor}
Let $Y_0$ and $Y_1$ be closed connected oriented three-manifolds and
set
\[
  V_i=\R\times Y_i,
  \qquad i=0,1.
\]
Every orientation-preserving, end-order-preserving homeomorphism
\[
  g\colon V_0\longrightarrow V_1
\]
induces a unique $\Z$-equivariant bijection
\[
  g^\sharp\colon\cJ(V_1)\longrightarrow\cJ(V_0)
\]
preserving $\Gamma$ and $\widetilde\Theta$ after the spin, $\Spinc$,
cohomological, and framed data have been transported by $g$.  These
bijections are contravariantly functorial.  If $g$ is a
diffeomorphism, then
\[
  g^\sharp[J]=[g^*J],
\]
where
\[
  (g^*J)_x=(dg_x)^{-1}\circ J_{g(x)}\circ dg_x
\]
is the geometric pullback.
\end{theorem}

\begin{proof}[Source and explanation]
This is the specialization of \cite[Theorem~4.7]{Gompf2022} to the
product manifolds $V_i=\R\times Y_i$.  We recall the structure of
Gompf's construction to clarify the uniqueness assertion used below.

Fix a spin structure $s_1$ on $V_1$.  By
\cite[Theorem~4.1]{Gompf2022}, the homeomorphism $g$ transports spin
and $\Spinc$ structures functorially; write
\[
  s_0=g^*s_1.
\]
For $[J_1]\in\cJ(V_1)$, transport the induced $\Spinc$ structure
$\mathfrak s_{J_1}$ to $V_0$.  The defining identity for $\Gamma$
then requires the image class $[J_0]$ to satisfy
\[
  \Gamma(J_0,s_0)=g^*\Gamma(J_1,s_1).
\]
By \eqref{eq:Gamma-orbit-space}, this condition determines a unique
$\Z$-orbit in $\cJ(V_0)$.

Choose $\phi_1\in\Omega(V_1)$ dual to $\Gamma(J_1,s_1)$.  Under the
Pontryagin--Thom identification, precomposition by $g$ transports it
to a class
\[
  \phi_0=g^*\phi_1\in\Omega(V_0)
\]
dual to $g^*\Gamma(J_1,s_1)$.  It remains to select the correct
element of the target orbit.  By \eqref{eq:Gompf-Theta-coset}, the
possible values of $\widetilde\Theta$ on an orbit form an index-four
coset determined by the Rohlin invariant of the corresponding spin
three-manifold.  The topological invariance modulo $8$ supplied by
\cite[Proposition~4.4]{Gompf2022} implies that the source and target
orbits have the same allowed coset.  Since the positive generator of
the $\Z$-action changes $\widetilde\Theta$ by $-4$ according to
\eqref{eq:Gompf-Theta-action}, there is a unique $\Z$-equivariant
bijection between these orbits preserving $\widetilde\Theta$.
Applying this construction to every orbit defines $g^\sharp$.

Gompf proves that the construction is independent of the auxiliary
spin and framed choices and that it is compatible with identity maps
and composition; see \cite[Proposition~4.2 and Theorem~4.7]{Gompf2022}.
Uniqueness then gives contravariant functoriality.
If $g$ is a diffeomorphism, geometric pullback transports all of the
spin, $\Spinc$, cohomological, and framed data above and preserves the
relative Euler number, Euler characteristic, and signature occurring
in \eqref{eq:Gompf-Theta-definition}.  It therefore has the defining
properties of $g^\sharp$, so uniqueness gives
\[
  g^\sharp[J]=[g^*J].
\]
\end{proof}

\begin{corollary}[Proper-homotopy invariance of pullback on $\cJ$]
\label{prop:J-proper-invariance}
Let $Y_0$ and $Y_1$ be closed connected oriented three-manifolds and
set
\[
  V_i=\R\times Y_i,
  \qquad i=0,1.
\]
Suppose that
\[
  g_0,g_1\colon V_0\longrightarrow V_1
\]
are orientation-preserving, end-order-preserving diffeomorphisms that
are properly homotopic, not necessarily isotopic through
diffeomorphisms.  Then
\[
  g_0^*=g_1^*\colon\cJ(V_1)\longrightarrow\cJ(V_0).
\]
\end{corollary}

\begin{proof}
Let
\[
  H\colon[0,1]\times V_0\longrightarrow V_1
\]
be a proper homotopy from $g_0$ to $g_1$.  The two maps are homotopic,
so they induce the same pullback on ordinary cohomology:
\[
  g_0^*=g_1^*\colon H^2(V_1;\Z)\longrightarrow H^2(V_0;\Z).
\]
By \cite[Theorem~4.1]{Gompf2022}, properly homotopic
orientation-preserving proper homotopy equivalences also induce the
same pullback on spin and $\Spinc$ structures.  Thus
\[
  g_0^*=g_1^*\colon\Spin(V_1)\longrightarrow\Spin(V_0)
\]
and
\[
  g_0^*=g_1^*\colon\Spinc(V_1)\longrightarrow\Spinc(V_0).
\]
Here $\Spin(V)$ and $\Spinc(V)$ denote the corresponding sets of
structures.  The defining identity
\[
  \mathfrak s_J=\mathfrak s_s+\Gamma(J,s)
\]
then shows that $g_0$ and $g_1$ transport $\Gamma$ identically.

The framed data agree as well.  Under the Pontryagin--Thom
identification $\Omega(V_i)\cong[V_i,S^2]$, pullback is
precomposition.  Hence, for every $\phi\in\Omega(V_1)$,
\[
  g_0^*\phi=g_1^*\phi,
\]
because $g_0$ and $g_1$ are homotopic.

Fix $[J]\in\cJ(V_1)$, a spin structure $s_1$, and a class
$\phi_1\in\Omega(V_1)$ dual to $\Gamma(J,s_1)$.  The two functorial
images
\[
  g_0^\sharp[J],
  \qquad
  g_1^\sharp[J]
\]
have the same value of $\Gamma$ and therefore belong to the same
$\Z$-orbit in $\cJ(V_0)$.  They also have the same transported spin,
$\Spinc$, cohomological, and framed data.  Since each $g_k^\sharp$
preserves $\widetilde\Theta$, the two classes have the same value of
that secondary invariant.  Thus $g_0^\sharp$ and $g_1^\sharp$ are $\Z$-equivariant
bijections with exactly the same transported spin, $\Spinc$,
cohomological, and framed data, and both preserve the corresponding
values of $\Gamma$ and $\widetilde\Theta$.  The uniqueness assertion
in \cref{thm:Gompf-homeomorphism-functor} therefore gives
\[
  g_0^\sharp[J]=g_1^\sharp[J].
\]
This holds for every $[J]$, so $g_0^\sharp=g_1^\sharp$.  Finally,
because $g_0$ and $g_1$ are diffeomorphisms,
\cref{thm:Gompf-homeomorphism-functor} identifies $g_k^\sharp$ with
geometric pullback $g_k^*$.  Hence $g_0^*=g_1^*$ on $\cJ(V_1)$.
\end{proof}

\begin{remark}\label{rem:no-differentiate}
If the proper homotopy were an isotopy through diffeomorphisms $g_t$,
the conclusion would follow directly from the homotopy
$J_t=g_t^*J$.  A general proper homotopy passes through arbitrary
proper maps, whose differentials need not be bundle isomorphisms, and
a homotopy of the underlying maps need not determine a homotopy from
$dg_0$ to $dg_1$.  Thus one cannot differentiate the given homotopy
to obtain a homotopy from $g_0^*J$ to $g_1^*J$.  Gompf's topological
functoriality is precisely what bypasses this tangential obstruction.
\end{remark}

\section{Proof of formal symplectization rigidity}\label{sec:formal-proof}

We now assemble the symplectic, differential-topological, and almost-complex inputs.

\begin{proof}[Proof of \cref{thm:formal-rigidity}]
Choose positive contact forms $\alpha_i$ with $\ker\alpha_i=\xi_i$, and write
\[
  V_i=\R\times Y_i,
  \qquad
  \omega_i=d(e^t\alpha_i).
\]
Let
\[
  \Phi\colon(V_0,\omega_0)\longrightarrow(V_1,\omega_1)
\]
be a symplectomorphism.

By \cref{prop:end-preservation}, $\Phi$ preserves the order of the two ends.  It is orientation preserving because it is symplectic.  Therefore \cref{prop:product-normal-form} gives an orientation-preserving diffeomorphism
\[
  f\colon Y_0\longrightarrow Y_1
\]
and an end-order-preserving proper homotopy from $\Phi$ to
\[
  F:=\id_\R\times f.
\]

Choose adapted compatible almost-complex structures $J_i$ on $(V_i,\omega_i)$ as in \eqref{eq:adapted-J}.  Since $\Phi$ is symplectic, $\Phi^*J_1$ is $\omega_0$-compatible.  The space of $\omega_0$-compatible almost-complex structures is contractible, so
\begin{equation}\label{eq:compatible-equality}
  [J_0]=\Phi^*[J_1]\in\cJ(V_0).
\end{equation}
By this proper homotopy and \cref{prop:J-proper-invariance},
\begin{equation}\label{eq:replace-Phi-by-product}
  \Phi^*[J_1]=F^*[J_1]\in\cJ(V_0).
\end{equation}
Applying the product correspondence \cref{prop:product-correspondence}, its naturality, and \cref{lem:adapted-J}, we obtain
\begin{align*}
  [\xi_0]
  &=\rho_{Y_0}([J_0])\\
  &=\rho_{Y_0}(F^*[J_1])\\
  &=f^*\rho_{Y_1}([J_1])\\
  &=[f^*\xi_1].
\end{align*}
Thus $\xi_0$ and $f^*\xi_1$ are homotopic as oriented plane fields.
\end{proof}

\section{Proof of overtwisted symplectization rigidity}\label{sec:overtwisted-proof}

We isolate the deep contact-topological input and prove the stability step in full.

\begin{theorem}[Eliashberg's overtwisted classification]\label{thm:Eliashberg-overtwisted}
Let $Y$ be a closed oriented three-manifold.  The inclusion of the
space of positive cooriented overtwisted contact structures into the
space of positive oriented two-plane fields induces a bijection on
path components.  Equivalently, two positive cooriented overtwisted
contact structures on $Y$ are isotopic through positive contact
structures if and only if they are homotopic as oriented plane fields.
\end{theorem}

This is Eliashberg's classification theorem \cite{Eliashberg1989}; see also Huang's convex-surface proof \cite{Huang2013}.

\begin{proof}[Proof of \cref{thm:main-overtwisted}]
Suppose first that
\[
  S(Y_0,\xi_0)\congsymp S(Y_1,\xi_1).
\]
By \cref{thm:formal-rigidity}, there is an orientation-preserving diffeomorphism $f\colon Y_0\to Y_1$ such that
\[
  \xi_0\simeq f^*\xi_1
\]
as oriented plane fields on $Y_0$.  Since $\xi_1$ is overtwisted, so is $f^*\xi_1$.  By \cref{thm:Eliashberg-overtwisted}, the contact structures $\xi_0$ and $f^*\xi_1$ are joined by a path of positive contact structures.  Applying Gray stability \cite{Gray1959} gives a diffeomorphism $\psi\in\Diff_0(Y_0)$ such that
\[
  \psi^*f^*\xi_1=\xi_0.
\]
Therefore $f\circ\psi\colon(Y_0,\xi_0)\to(Y_1,\xi_1)$ is a coorientation-preserving contactomorphism.

Conversely, every coorientation-preserving contactomorphism lifts to a strict exact symplectomorphism by \cref{lem:strict-lift}.  This proves both implications in \eqref{eq:main-iff}.
\end{proof}

\section{Higher dimension}\label{sec:scope}

The higher-dimensional counterpart of
\cref{thm:main-overtwisted} is false.

\begin{proposition}[Higher-dimensional overtwisted counterexamples]
\label{prop:higher-dimensional-overtwisted}
For every odd integer $d\ge 5$, there are closed connected cooriented
overtwisted contact $d$-manifolds
$(M_0,\widehat\xi_0)$ and $(M_1,\widehat\xi_1)$ that are not
contactomorphic---indeed, $M_0$ and $M_1$ are not diffeomorphic---although
their symplectizations are exact symplectomorphic.
\end{proposition}

\begin{proof}
Write
\[
d=2n+3,\qquad n\ge 1,
\]
and set
\[
M_0=L(7,1)\times S^{2n},
\qquad
M_1=L(7,2)\times S^{2n},
\qquad n\ge1.
\]
These manifolds are $h$-cobordant and are not diffeomorphic.
For the $h$-cobordism statement, Milnor proves more generally that if
$M_1$ and $M_2$ are closed parallelizable $k$-manifolds of the same
homotopy type and $N>k>1$, then
\[
M_1\times S^{N-1}
\quad\text{and}\quad
M_2\times S^{N-1}
\]
are $h$-cobordant
\cite[\S2, p.~579]{Milnor1961}.
Since $L(7,1)$ and $L(7,2)$ are parallelizable homotopy-equivalent
$3$-manifolds, this applies with $N=2n+1$ whenever $n\ge2$.
The remaining case $n=1$ is treated separately by Milnor, who proves
that
\[
L(7,1)\times S^2
\quad\text{and}\quad
L(7,2)\times S^2
\]
are $h$-cobordant
\cite[\S2, p.~580]{Milnor1961}.
Finally, Milnor's torsion computation shows that
\[
L(7,1)\times S^{m}
\not\cong_{\mathrm{diff}}
L(7,2)\times S^{m}
\]
for every even $m$
\cite[Corollary~2, p.~588]{Milnor1961}.
Thus $M_0$ and $M_1$ are $h$-cobordant but not diffeomorphic.
This particular family is also recalled explicitly by Courte in
\cite[Introduction, pp.~1--2]{Courte2014}.

Fix an $h$-cobordism
\[
W\colon M_0\longrightarrow M_1
\]
as above.  Courte observes that $M_0$ admits a contact structure; see
\cite[Introduction, pp.~1--2]{Courte2014}.  Fix one such contact
structure $\xi_0$.  By
\cite[Lemma~4.1, pp.~8--9]{Courte2014}, the $h$-cobordism $W$
admits a flexible Weinstein structure inducing $\xi_0$ on its negative
boundary and some contact structure $\xi_1$ on its positive boundary.
Thus we have a flexible Weinstein $h$-cobordism
\[
W\colon (M_0,\xi_0)\longrightarrow(M_1,\xi_1).
\]

We next introduce overtwistedness without changing the underlying
smooth manifolds.  The standard sphere $S^d$ carries an almost contact
structure, for instance the one underlying its standard contact
structure.  By the overtwisted $h$-principle of
Borman--Eliashberg--Murphy, more precisely
\cite[Corollary~1.3]{BormanEliashbergMurphy2015}, this almost contact
structure is homotopic to an overtwisted contact structure.  Fix such
a structure and denote it by $\xi_{\mathrm{ot}}$.

Choose a contact form $\alpha_{\mathrm{ot}}$ for $\xi_{\mathrm{ot}}$.
The product
\[
P=[0,1]\times S^d
\]
carries the trivial Weinstein-cobordism structure obtained by restricting
the symplectization:
\[
\lambda_P=e^t\alpha_{\mathrm{ot}},
\qquad
\omega_P=d\lambda_P,
\qquad
X_P=\partial_t.
\]
Thus $P$ is a Weinstein cobordism from
$(S^d,\xi_{\mathrm{ot}})$ to itself, with no critical points.

We now form the Weinstein connected sum of $W$ and $P$ along
neighborhoods of properly embedded arcs joining the negative to the
positive boundary.  This is precisely the local connected-sum operation
used by Courte in the proof of
\cite[Corollary~2.2, p.~665]{Courte2016}: there the two Weinstein
cobordisms are glued along neighborhoods of arcs running from
$\partial_-$ to $\partial_+$, and Courte observes that the resulting
Weinstein cobordism is flexible.  Applying the same construction here
gives a flexible Weinstein cobordism
\[
\widehat W\colon
(M_0\mathbin{\#}S^d,\,
 \xi_0\mathbin{\#}\xi_{\mathrm{ot}})
\longrightarrow
(M_1\mathbin{\#}S^d,\,
 \xi_1\mathbin{\#}\xi_{\mathrm{ot}}).
\]

Let us also check explicitly that the underlying smooth cobordism
$\widehat W$ is still an $h$-cobordism.  Choose the arc in
$P=S^d\times[0,1]$ to be a product arc $\{p\}\times[0,1]$.  A regular
neighborhood of this arc is diffeomorphic to
$D^d\times[0,1]$, while
\[
\overline{
P\setminus \nu(\{p\}\times[0,1])
}
\cong
\bigl(S^d\setminus\operatorname{int}D^d\bigr)\times[0,1]
\cong
D^d\times[0,1].
\]
Consequently, on the smooth level the connected-sum operation merely
removes a copy of $D^d\times[0,1]$ from $W$ and glues back another
copy of $D^d\times[0,1]$.  Hence $\widehat W$ is diffeomorphic to $W$
as an unparametrized cobordism.  In particular, $\widehat W$ is again
an $h$-cobordism.  Its boundary components are, as expected,
\[
M_i\mathbin{\#}S^d\cong M_i,
\qquad i=0,1.
\]

It remains to verify that the two boundary contact structures are
overtwisted.  By
\cite[Definition~3.5]{BormanEliashbergMurphy2015}, overtwistedness
means precisely that the contact manifold admits a contact embedding
of a Borman--Eliashberg--Murphy overtwisted disc.  Fix such an
overtwisted disc
\[
D_{\mathrm{ot}}\subset(S^d,\xi_{\mathrm{ot}}).
\]
Choose the Darboux ball in the $S^d$-factor used for the contact
connected sum disjoint from a neighborhood of $D_{\mathrm{ot}}$.
The connected-sum construction changes the contact structure only in
the chosen Darboux balls and in the connecting neck.  Therefore the
contact germ along $D_{\mathrm{ot}}$ is unchanged, and
$D_{\mathrm{ot}}$ survives as an overtwisted disc in each contact
connected sum
\[
(M_i\mathbin{\#}S^d,\,
 \xi_i\mathbin{\#}\xi_{\mathrm{ot}}).
\]
Hence both of these contact structures are overtwisted.

Finally, Courte's symplectic Mazur-trick argument applies to any
flexible Weinstein $h$-cobordism between closed contact manifolds of
dimension at least five.  This consequence is stated explicitly in
\cite[Introduction, p.~657]{Courte2016}; the original construction is
the proof of \cite[Theorem~4.3, pp.~9--12]{Courte2014}.  In that proof,
Courte chooses an inverse $h$-cobordism, forms an infinite alternating
concatenation of the two Weinstein cobordisms, and uses a Weinstein
homotopy together with a Moser argument to identify the resulting exact
symplectic manifold with the two symplectizations.  Applied to
$\widehat W$, it gives an exact symplectomorphism
\[
S\bigl(M_0\mathbin{\#}S^d,\,
       \xi_0\mathbin{\#}\xi_{\mathrm{ot}}\bigr)
\cong_{\mathrm{ex}}
S\bigl(M_1\mathbin{\#}S^d,\,
       \xi_1\mathbin{\#}\xi_{\mathrm{ot}}\bigr).
\]

Using the standard diffeomorphisms
\[
M_i\mathbin{\#}S^d\cong M_i,
\]
define
\[
\widehat\xi_i
:=
\xi_i\mathbin{\#}\xi_{\mathrm{ot}},
\qquad i=0,1.
\]
The contact structures $\widehat\xi_0$ and $\widehat\xi_1$ are
overtwisted by the preceding argument, their symplectizations are exact
symplectomorphic, while their underlying manifolds $M_0$ and $M_1$ are
not diffeomorphic by Milnor's Reidemeister-torsion calculation.
Therefore they cannot be contactomorphic.
\end{proof}
Thus \cref{thm:main-overtwisted} is a genuinely low-dimensional rigidity statement.

\section*{Acknowledgements}

This work was supported by JSPS KAKENHI Grant Number JP26K16990
(Grant-in-Aid for Early-Career Scientists) and by the Start-up Fund of
the Kavli Institute for the Physics and Mathematics of the Universe
(Kavli IPMU), The University of Tokyo.

This paper grew out of discussions between the author and ChatGPT.
The author used ChatGPT for exploratory bibliographic searches, checks
of algebraic and sign computations, and English and \LaTeX{} editing.
ChatGPT was not treated as an authoritative source: the author
independently checked the cited literature and all mathematical
arguments, computations, and references, and takes full responsibility
for the content of the manuscript.

\bibliographystyle{amsplain}
\bibliography{Symplectization_final}

\end{document}